\input ./style/arxiv-vmsta.cfg
\documentclass[numbers,compress,v1.0.1]{vmsta}
\usepackage{vtexbibtags}
\usepackage{authorquery}
\usepackage{mathrsfs}
\usepackage{mathbh}

\volume{5}
\issue{4}
\pubyear{2018}
\firstpage{445}
\lastpage{456}
\aid{VMSTA114}
\doi{10.15559/18-VMSTA114}
\articletype{research-article}

\startlocaldefs
\newcommand{\rright}{\right}
\newcommand{\lleft}{\left}

\numberwithin{equation}{section}
\newcommand{\coloneqq}{:=}

\allowdisplaybreaks

\newtheorem{thm}{Theorem}
\theoremstyle{definition}
\newtheorem{remark}{Remark}
\endlocaldefs

\begin{aqf}
\querytext{Q1}{This 'instead' is not informative here; 'instead of (\label\string{eq:frac-eq-u\string})' would be better, if it is correct after all. Otherwise, this 'instead' could be omitted.}
\end{aqf}
\begin{document}

\begin{frontmatter}
\pretitle{Research Article}

\title{Drifted Brownian motions governed by fractional tempered derivatives}

\author[a]{\inits{M.}\fnms{Mirko}~\snm{D'Ovidio}\thanksref{cor1}\ead
[label=e1]{mirko.dovidio@uniroma1.it}}
\author[b]{\inits{F.}\fnms{Francesco}~\snm{Iafrate}\ead
[label=e2]{francesco.iafrate@uniroma1.it}}
\author[b]{\inits{E.}\fnms{Enzo}~\snm{Orsingher}\ead
[label=e1]{enzo.orsingher@uniroma1.it}}

\thankstext[type=corresp,id=cor1]{Corresponding author.}
\address[a]{SBAI, \institution{Sapienza University of Rome}, \cny{Italy}}
\address[b]{DSS, \institution{Sapienza University of Rome}, \cny{Italy}}




\markboth{M. D'Ovidio et al.}{Drifted Brownian motions governed by
fractional tempered derivatives}

\begin{abstract}
Fractional equations governing the
distribution of reflecting drifted Brownian motions are presented. The equations are
expressed in terms of tempered Riemann--Liouville type derivatives. For
these operators a Marchaud-type form is obtained and a Riesz tempered
fractional derivative is examined, together with its Fourier transform.
\end{abstract}

\received{\sday{24} \smonth{4} \syear{2018}}
\revised{\sday{28} \smonth{6} \syear{2018}}
\accepted{\sday{29} \smonth{8} \syear{2018}}
\publishedonline{\sday{19} \smonth{9} \syear{2018}}
\end{frontmatter}

\section{Introduction}
In this paper we consider various forms of tempered fractional
derivatives. For a function $ f $ continuous and compactly supported on the
positive real line, let us consider the Marchaud type operator defined by
%
\begin{equation}
\label{eq:march-def} \bigl(\mathscr{D}^{\alpha, \eta}f\bigr) (x) =
\int
_{0}^{\infty} \bigl(f(x) - f(x-y) \bigr) \, \varPi(
\mathrm{d} y)
\end{equation}
where
%
\begin{equation}
\varPi(\mathrm{d} y) = \frac{\alpha}{\varGamma(1-\alpha)} \frac
{e^{-\eta y }}{y^{\alpha+1}} dy ,\quad  y>0\xch{,}{.}
\label{LevyMH}
\end{equation}
with $ \eta>0, 0 < \alpha<1 $. The operator \eqref{eq:march-def}
coincides with the classical Marchaud derivative for $ \eta= 0 $.

The Laplace transform of the fractional operator \eqref{eq:march-def} reads
%
\begin{align}
\int_{0}^{\infty} e^{ - \lambda x} \, \bigl(
\mathscr{D}^{\alpha, \eta} f \bigr) (x) \, \mathrm{d} x &= \Biggl(
\int
_{0}^{\infty} \bigl( 1 - e^{ - \lambda y} \bigr) \varPi(
\mathrm{d} y) \Biggr) \tilde{f} (\lambda)
\nonumber
\\
&= \bigl( (\eta+ \lambda)^{\alpha}- \eta^{\alpha} \bigr) \tilde{f}(
\lambda) . \label{symbolH}
\end{align}
Throughout the work we denote by $\tilde{f}$ the Laplace transform of
$f$. In the Fourier analysis the factor $ (\eta+ i \lambda)^{\alpha
}- \eta^{\alpha}$ is the multiplier of the Fourier transform of
$ f $~\xch{\cite{meersc15}}{(\cite{meersc15})}. Tempered fractional derivatives emerge in the study
of equations driving the tempered subordinators \xch{\cite{beghin,meersc15}}{(\cite{beghin,meersc15})}.
In particular, the operator \eqref{eq:march-def} is the generator of
the subordinator $ H_{t}, t>0 $, with L\'{e}vy measure \eqref{LevyMH}
and density law whose Laplace transform is given by \eqref{symbolH},
that is,
\begin{align}
\mathbb{E} e^{- \lambda H_{t}} &= e^{-t ( (\eta+ \lambda)^{\alpha}-
\eta^{\alpha} ) } = e^{ - t \int_{0}^{\infty} ( 1 - e^{-\lambda y})
\varPi(\mathrm{d} y)}, \quad\lambda
>0.
\nonumber
\end{align}
The process $ H_{t} $ is called relativistic subordinator and
coincides, for $ \eta= 0 $, with a
positively skewed L\'{e}vy process, that is a stable subordinator.
Tempered stable subordinators can be viewed as the limits of Poisson
random sums with tempered power law jumps \xch{\cite{meersc15}}{(\cite{meersc15})}.

The fractional operator $ \mathscr{D}^{\alpha, \eta} f$ defined in
\eqref{eq:march-def} is related
to the tempered upper Weyl derivatives defined by
%
\begin{equation}
\label{eq:up-weyl-intro} \bigl(\hat{\mathscr{D}}^{\alpha, \eta
}_{+}f\bigr) (x) =
\frac{1}{\varGamma(1-\alpha)} \frac{\mathrm{d}}{\mathrm{d} x} \int
_{-\infty}^{x}
\frac{f(t)}{(x-t)^{\alpha}} e^{- \eta(x-t)} \, \mathrm{d} t.
\end{equation}
By combining \eqref{eq:up-weyl-intro} with the lower Weyl tempered
derivatives we obtain
the Riesz tempered fractional derivatives $ \frac{\partial^{\alpha,
\eta} f}{\partial|x|^{\alpha}} $ 
from which we obtain the explicit
Fourier transform in \eqref{eq:temp-riesz-fou}.

We consider the Dzherbashyan--Caputo derivative of order $ \frac{1}{2} $,
that is,
%
\begin{equation}
\bigl( D^{\frac{1}{2}} f \bigr) (t) = \frac{1}{\sqrt{\pi}} \int
_{0}^{t} f'(s) (t-s)^{-\frac{1}{2}} \,
\mathrm{d} s
\end{equation}
with the Laplace transform
\begin{align*}
\int_{0}^{\infty}e^{-\lambda t} \bigl(
D^{\frac{1}{2}} f \bigr) (t)\, \mathrm{d} t = \lambda^{\frac{1}{2}}
\tilde{f}(
\lambda) - \lambda^{\frac{1}{2}-1} f(0), \quad\lambda>0.
\end{align*}
The relationship between the Riemann--Liouville and the
Dzherbashyan--Caputo de\-ri\-vative can be given as follows,
%
\begin{equation}
\bigl( \mathscr{D} ^{\frac{1}{2}} f \bigr) (t) = \bigl( D^{\frac{1}{2}} f
\bigr) (t) + \frac{t^{\frac{1}{2} -1}}{\varGamma(\frac{1}{2})} f(0),
\end{equation}
from which we observe that
%
\begin{align}
\int_{0}^{\infty}e^{-\lambda t} \bigl( \mathscr{D}
^{\frac{1}{2}} f \bigr) (t) \, \mathrm{d} t = \lambda^{\frac{1}{2}}
\tilde{f}(
\lambda). \label{Ltransf-RL}
\end{align}
We remark that the problems
\begin{align*}
\lleft\lbrace
\begin{array}{@{}ll@{}}
\displaystyle\bigl(D^{\frac{1}{2}} u\bigr)(t) = - \frac{\partial u}{\partial y},
\quad  t>0, y>0\\[3pt]
\displaystyle u(0,y)=\delta(y)
\end{array}
\rright. \quad\text{\textrm{and}} \quad\lleft\lbrace
\begin{array}{@{}ll@{}}
\displaystyle\bigl(\mathscr{D}^{\frac{1}{2}} u\bigr)(t) = - \frac{\partial
u}{\partial y}, &  t>0, y>0\\[3pt]
\displaystyle u(0,y)=\delta(y)\\[3pt]
\displaystyle u(t,0)= \frac{1}{\sqrt{\pi t}}, &  t>0
\end{array}
\rright.
\end{align*}
have a unique solution given by the density law of an inverse to a
stable subordinator, say $L_{t}$ (see for example \cite[formulas 3.4
and 3.5]{dovidio}). It is well known that $L_{t}$ (with $L_{0}=0$) is
identical in law to a folded Brownian motion $|B_{t}|$ (with
$B_{0}=0$), that is, $u$ is the unique solution to the problem
\begin{align*}
\lleft\lbrace
\begin{array}{@{}ll@{}}
\displaystyle\frac{\partial u}{\partial t} = \frac{\partial^{2}
u}{\partial y^{2}}, \quad  t>0,\, y>0,\\[3pt]
\displaystyle u(0, y) = \delta(y),\\[3pt]
\displaystyle\frac{\partial u}{\partial y}(t, 0) =0.
\end{array}
\rright.
\end{align*}
Thus, by considering the theory of time changes, there exist
interesting connections between fractional Cauchy problems and the
domains of the generators of the base processes. In our view,
concerning the drifted Brownian motion, the present paper gives new
results also in this direction.

We denote by
%
\begin{equation}
\label{eq:temp-rl-def} \mathscr{D} _{t}^{\frac{1}{2}, \eta}f \coloneqq
e^{-\eta t} \mathscr{D} _{t}^{\frac{1}{2}} \bigl(
e^{\eta t} f \bigr) - \sqrt{\eta}f
\end{equation}
the tempered Riemann--Liouville type derivative. The equality between
definitions \eqref{eq:temp-rl-def} and \eqref{eq:march-def} can be
verified by comparing the corresponding Laplace transforms. Indeed,
from \eqref{Ltransf-RL},
\begin{align*}
\int_{0}^{\infty}e^{ - \lambda t}\, \mathscr{D}
_{t}^{\frac{1}{2}, \eta}f \, \mathrm{d} t ={}& \int_{0}^{\infty}e^{
-(\lambda+ \eta) t}
\mathscr{D} _{t}^{\frac{1}{2}} ( g ) \, \mathrm{d} t - \sqrt{\eta}
\tilde{f} (\lambda)
\\
={}& \sqrt{\lambda+\eta} \int_{0}^{\infty}e^{-(\lambda+\eta) t}
g(t)\, \mathrm{d} t - \sqrt{\eta}\tilde{f} (\lambda)
\end{align*}
where $g(t) = e^{\eta t} f(t)$.

Let $B $ represent a Brownian motion starting at the origin with
generator $\Delta$. In the paper we show that the transition density $
u = u(x,y,t) $ of the 1-dimensional process
\begin{equation*}
B^{\mu}(t) = B(t) + \mu t + x, \quad\mu>0, \, x \in\mathbb{R},
\end{equation*}
satisfies the fractional equation on $(0, \infty) \times\mathbb{R}^{2}$
%
\begin{equation}
\label{eq:frac-eq-u}
\lleft\{
\begin{aligned} &\mathscr{D} _{t}^{\frac{1}{2}, \eta} u +
\sqrt{\eta} \,u= a(x,y) \biggl( \frac{\partial u}{\partial x} + \sqrt
{\eta} u \biggr) = -
a(x,y) \biggl( \frac{\partial u}{\partial y} - \sqrt{\eta} u \biggr),
\\
&u(x,y,0) = \delta(x-y)
\end{aligned}
\rright.
\end{equation}
where
\begin{equation*}
a(x,y) = \mathbh{1}_{(-\infty,y] } (x) - \mathbh{1} _{(y, \infty) } (x)
\end{equation*}
and
\begin{align*}
\eta= \frac{\mu^{2}}{4}.
\end{align*}

A different result concerns the reflected process
\begin{equation*}
|B^{\mu}(t) + \mu t| + x = |\hat{B}|^{\mu}(t)
\end{equation*}
whose transition density $v = v(x,y,t) $ satisfies the equation
%
\begin{equation}
\label{eq:frac-eq-v} \mathscr{D} _{t}^{\frac{1}{2}, \eta} v + \sqrt
{\eta} \, v=
\frac{\partial v}{\partial x} + \sqrt{\eta} \tanh\bigl(\sqrt{\eta}
(y-x) \bigr) v,
\quad t >0,\; y > x >0,
\end{equation}
with initial and boundary conditions
\begin{align*}
v(x,y,0) ={}& \delta( y - x),
\\
v(x,x,t) ={}& \frac{e^{-\eta t}}{\sqrt{\pi t}}, \quad t>0,
\end{align*}
and
\begin{align*}
\eta= \frac{\mu^{2}}{4}.
\end{align*}
The fractional equation governing the iterated Brownian motion $ B^{\mu
_{2}} ( | B^{\mu_{1}} (t) | )$\break ($ B^{\mu_{j}} $ being independent)
has been studied in \cite{iafrate18} and in the special case $ B^{\mu
}( | B (t) | ) $ explicitly derived. For the iterated Bessel process a
similar analysis is performed in~\cite{DOVIDIO2011441}. A~general
presentation of tempered fractional calculus can be found in the paper \cite{meersc15}.

Many processes like Brownian motion, iterated Brownian motion, Cauchy
process have transition functions
satisfying different partial differential equations and also are
solutions of fractional equations
of different forms with various fractional derivatives. We here show
that a similar situation arises when
drifted reflecting Brownian motion is considered but in this case the
corresponding fractional equations
involve tempered Riemann--Liouville type derivatives.

\section{A generalization of the tempered Marchaud derivative}
In this section we study the tempered Weyl derivatives (upper and lower
ones) and construct the Riesz tempered derivative.
We are able to obtain the Fourier transform of the Riesz tempered
derivatives and thus to solve some generalized fractional diffusion
equation.

We start by giving the explicit forms of the tempered Weyl derivative
\begin{align}
& \bigl(\hat{\mathscr{D}}^{\alpha, \eta}_{+}f\bigr) (x)
\nonumber
\\
&= \frac{1}{\varGamma(1-\alpha)} \frac{\mathrm{d}}{\mathrm{d} x} \int
_{-\infty}^{x}
\frac{f(t)}{(x-t)^{\alpha}} e^{- \eta(x-t)} \, \mathrm{d} t
\\
&= \frac{1}{\varGamma(1-\alpha)} \frac{\mathrm{d}}{\mathrm{d} x} \int
_{0}^{\infty}
\frac{f(x-t)}{t^{\alpha}} e^{-\eta t} \, \mathrm{d} t
\nonumber
\\
&= \frac{1}{\varGamma(1-\alpha)} \frac{\mathrm{d}}{\mathrm{d} x} \int
_{0}^{\infty}
f(x-t) e^{-\eta t } \int_{t}^{\infty} \alpha
w^{-\alpha-1} \, \mathrm{d} w \, \mathrm{d} t
\nonumber
\\
&= \frac{1}{\varGamma(1-\alpha)} \int_{0}^{\infty} \alpha
w^{-\alpha-1} \, \mathrm{d} w \int_{0}^{w}
f'(x-t) e^{-\eta t} \, \mathrm{d} t
\nonumber
\\
&= \frac{1}{\varGamma(1-\alpha)} \int_{0}^{\infty} \alpha
w^{-\alpha-1} \, \mathrm{d} w \int_{x-w}^{x}
f'(t) e^{-\eta(x-t)} \, \mathrm{d} t
\nonumber
\\
&= \frac{1}{\varGamma(1-\alpha)} \int_{0}^{\infty} \alpha
w^{-\alpha-1} \, e^{-\eta x} \Biggl\{ f(t) e^{\eta t}
|_{x-w}^{x} - \eta\int_{x-w}^{x}
f(t) e^{\eta t} \, \mathrm{d} t \Biggr\} \, \mathrm{d} w
\nonumber
\\
&= \frac{1}{\varGamma(1-\alpha)} \int_{0}^{\infty} \alpha
w^{-\alpha-1} \, e^{-\eta x} \bigl[ f(x) e^{\eta x} -
f(x-w)e^{\eta(x-w)} \bigr] \, \mathrm{d} w
\nonumber
\\
& \quad- \frac{\eta}{\varGamma(1-\alpha)} \int_{0}^{\infty} \alpha
w^{-\alpha-1} \int_{x-w}^{x} f(t)
e^{-\eta(x-t)}\, \mathrm{d} w\, \mathrm{d} t
\nonumber
\\
&= \frac{1}{\varGamma(1-\alpha)} \int_{0}^{\infty} \alpha
w^{-\alpha-1} \, \bigl[ f(x) - f(x-w)e^{-\eta w} \bigr] \, \mathrm{d} w
\nonumber
\\
& \quad- \frac{\eta}{\varGamma(1-\alpha)} \int_{0}^{\infty} \alpha
w^{-\alpha-1} \int_{0}^{w} f(x-t)
e^{-\eta t}\, \mathrm{d} w \, \mathrm{d} t
\nonumber
\\
&= \frac{1}{\varGamma(1-\alpha)} \int_{0}^{\infty} \alpha
w^{-\alpha-1} \, \bigl[ f(x) + f(x) e^{-\eta w} - f(x) e^{-\eta w} -
f(x-w)e^{-\eta w} \bigr] \, \mathrm{d} w
\nonumber
\\
& \quad- \frac{\eta}{\varGamma(1-\alpha)} \int_{0}^{\infty} f(x-t)
e^{-\eta t} \int_{t}^{\infty}\alpha
w^{-\alpha-1} \, \mathrm{d} w \, \mathrm{d} t
\nonumber
\\
&= \frac{1}{\varGamma(1-\alpha)} \int_{0}^{\infty} \bigl(f(x) -
f(x-w) \bigr) \, \alpha\frac{e^{-\eta w}}{w^{\alpha+1}} \,\mathrm{d}
w \, \mathrm{d} t
\nonumber
\\
& \quad+ \frac{f(x)}{\varGamma(1-\alpha)} \int_{0}^{\infty} \alpha
w^{-\alpha-1} \bigl(1-e^{-\eta w} \bigr)\, \mathrm{d} w -
\frac{\eta}{\varGamma(1-\alpha)} \int_{0}^{\infty} f(x-t)
\frac{e^{-\eta t}}{t^{\alpha}} \, \mathrm{d} t
\nonumber
\\
&= \int_{0}^{\infty} \bigl(f(x) - f(x-w) \bigr) \varPi
(\mathrm{d} w) + \eta\int_{0}^{\infty} \bigl(f(x) -
f(x-w) \bigr) \frac{e^{-\eta w }}{w^{\alpha}\varGamma(1-\alpha)} \,
\mathrm{d} w
\nonumber
\end{align}
The derivative $ \hat{\mathscr{D}}_{+} ^{\alpha, \eta}$ can be
expressed in terms of $ \mathscr{D}^{\alpha, \eta} $ as follows:
\begin{equation*}
\hat{\mathscr{D}}^{\alpha, \eta}_{+}f = \mathscr{D}^{\alpha, \eta}
f -
\eta\, \mathscr{D}^{\alpha-1, \eta} f.
\end{equation*}
In the same way we can obtain the upper Weyl derivative in the Marchaud
form as
\begin{align}
& \bigl(\hat{\mathscr{D}}^{\alpha, \eta}_{-}f\bigr) (x)
\nonumber
\\
&= \frac{1}{\varGamma(1-\alpha)} \frac{\mathrm{d}}{\mathrm{d} x} \int
_{x}^{\infty}
\frac{f(t)}{(x-t)^{\alpha}} e^{- \eta(x-t)} \, \mathrm{d} t
\\
&= \int_{0}^{\infty} \frac{\alpha w ^{-\alpha-1}}{\varGamma(1-\alpha
)} \bigl\{
e^{-\eta w} f(x+w) - f(x)\bigr\} \,\mathrm{d} w + \eta\int
_{0}^{\infty} \frac{e^{\eta w} f(x+w)}{\varGamma(1-\alpha) w^{\alpha
} } \,\mathrm{d} w
\nonumber
\\
&= \frac{1}{\varGamma(1-\alpha)} \int_{0}^{\infty} \bigl[ f(x+w) -
f(x)\bigr] \frac{\alpha e^{-\eta w}}{w^{\alpha+1} }
\nonumber
\\
& \quad+ \frac{f(x)}{\varGamma(1-\alpha)} \int_{0}^{\infty} \alpha
w^{-\alpha-1} \bigl(e^{-\eta w } -1\bigr) \, \mathrm{d} w + \eta\int
_{0}^{\infty} f(x+t) \frac{e^{-\eta w}}{\varGamma(1-\alpha)
t^{\alpha}} \, \mathrm{d} t
\nonumber
\\
&= \int_{0}^{\infty} \bigl[ f(x+w) - f(x)\bigr] \varPi
(\mathrm{d} w) + \eta\int_{0}^{\infty} \bigl[ f(x+w) -
f(x)\bigr] \frac{e^{-\eta w}}{\varGamma(1-\alpha) w^{\alpha}} \,
\mathrm{d} w .
\nonumber
\end{align}
For $ 0 < \alpha<1 $ the Riesz fractional derivative writes
\begin{align}
\label{eq:riesz-der} \frac{\partial^{\alpha}f}{\partial|x|^{\alpha
}} &= - \frac{1}{2 \cos\frac{\alpha\pi}{2} \, \varGamma(1-\alpha)}
\int
_{-\infty}^{+\infty} \frac{f(t)}{|x-t|^{\alpha}} \, \mathrm{d} t
\\
&= - \frac{1}{2 \cos\frac{\alpha\pi}{2} \, \varGamma(1-\alpha)}
\Biggl[ \frac{\mathrm{d}}{\mathrm{d} x} \int_{-\infty}^{x}
\frac{f(t)}{(x-t)^{\alpha}} \, \mathrm{d} t - \frac{\mathrm{d}}{\mathrm
{d} x} \int_{x }^{\infty}
\frac{f(t)}{(t-x)^{\alpha}} \, \mathrm{d} t \Biggr].
\nonumber
\end{align}
In the same way we define the tempered Riesz derivative as
\begin{align*}
\frac{\partial^{\alpha, \eta} f}{\partial|x|^{\alpha}} = C_{\alpha
, \eta} \Biggl[ \frac{\mathrm{d}}{\mathrm{d} x} \int
_{-\infty}^{x} \frac{f(t)}{(x-t)^{\alpha}} \frac{e^{- \eta(x-t)} }{
\varGamma(1-\alpha)}\,
\mathrm{d} t - \frac{\mathrm{d}}{\mathrm{d} x} \int_{x }^{\infty}
\frac{f(t)}{(t-x)^{\alpha}} \frac{e^{- \eta(t-x)} }{ \varGamma
(1-\alpha)} \, \mathrm{d} t \Biggr]
\end{align*}
where $ C_{\alpha,\eta} $ is a suitable constant which will be
defined below.
In view of the previous calculations we have that
\begin{align}
\frac{\partial^{\alpha, \eta} f}{\partial|x|^{\alpha}} &=
C_{\alpha, \eta} \Biggl[  \int_{0}^{\infty}
\bigl(f(x) - f(x-w) \bigr) \frac{\alpha e^{-\eta w} \mathrm{d}
w}{\varGamma(1-\alpha) w^{\alpha+ 1} }
\nonumber
\\
&\quad  + \eta\int_{0}^{\infty} \bigl(f(x) - f(x-w) \bigr)
\frac{e^{-\eta w}\mathrm{d} w}{\varGamma(1-\alpha) w^{\alpha}}
\\
&\quad  - \int_{0}^{\infty} \bigl(f(x+w) - f(x) \bigr)
\frac{\alpha e^{-\eta w}\mathrm{d} w}{\varGamma(1-\alpha) w^{\alpha+
1} }
\nonumber
\\
&\quad  - \eta\int_{0}^{\infty} \bigl(f(x+w) - f(x) \bigr)
\frac{e^{-\eta w}\mathrm{d} w}{\varGamma(1-\alpha) w^{\alpha}}
\Biggr]
\nonumber
\\
& = C_{\alpha, \eta} \Biggl[ \int_{0}^{\infty} \bigl( 2
f(x) - f(x-w) - f(x+w) \bigr) \frac{\alpha e^{-\eta w}\mathrm{d}
w}{\varGamma(1-\alpha) w^{\alpha+ 1} }
\nonumber
\\
&\quad+ \eta\int_{0}^{\infty} \bigl( 2 f(x) - f(x-w)
- f(x+w) \bigr) \frac{ e^{-\eta w}\mathrm{d} w}{\varGamma(1-\alpha)
w^{\alpha} } \Biggr].
\nonumber
\end{align}
We now evaluate the Fourier transform of the tempered Riesz derivative
\begin{align}
\label{eq:temp-riesz-fou} \int_{-\infty}^{+\infty} e^{i\gamma x}
\frac{\partial^{\alpha, \eta} f}{\partial|x|^{\alpha}} \, \mathrm
{d} x ={}& C_{\alpha, \eta} \Biggl\{  \hat{F}(\gamma
) \int_{0}^{\infty} \bigl( 1 - e^{i\gamma w} \bigr)
\frac{\alpha e^{-\eta w}\mathrm{d} w}{\varGamma(1-\alpha) w^{\alpha+
1} }
\\
& +  \eta\hat{F}(\gamma) \int_{0}^{\infty} \bigl( 1 -
e^{i\gamma w} \bigr) \frac{ e^{-\eta w}\mathrm{d} w}{\varGamma
(1-\alpha) w^{\alpha} }
\nonumber
\\
& -  \hat{F}(\gamma) \int_{0}^{\infty} \bigl(
e^{i\gamma w} -1 \bigr) \frac{\alpha e^{-\eta w}\mathrm{d} w}{\varGamma
(1-\alpha) w^{\alpha+ 1} }
\nonumber \\
& -  \eta\hat{F}(\gamma) \int_{0}^{\infty} \bigl(
e^{i\gamma w} -1 \bigr) \frac{ e^{-\eta w}\mathrm{d} w}{\varGamma
(1-\alpha) w^{\alpha} }  \Biggr\}
\nonumber
\\
={}&  C_{\alpha, \eta} \hat{F} ( \gamma) \Biggl\{  2 \int
_{0}^{\infty} (1 - \cos\gamma w) \frac{\alpha e^{-\eta w}\mathrm{d}
w}{\varGamma(1-\alpha) w^{\alpha+ 1} }
\nonumber
\\
& + 2 \eta\int_{0}^{\infty} (1 - \cos\gamma w)
\frac{ e^{-\eta w}\mathrm{d} w}{\varGamma(1-\alpha) w^{\alpha} }
\Biggr\}
\nonumber
\\
={}& C_{\alpha, \eta} \hat{F} ( \gamma) \Biggl\{ -2 w^{-\alpha}
(1-\cos
\gamma w) \frac{e^{-\eta w}}{\varGamma(1-\alpha)} |_{0}^{\infty}
\nonumber
\\
& - 2 \eta\int_{0}^{\infty} (1 - \cos\gamma w)
\frac{e^{-\eta w} \mathrm{d} w}{w^{\alpha}\varGamma(1-\alpha)}
\nonumber
\\
& + 2 \gamma\int_{0}^{\infty} \frac{e^{-\eta w} \sin\gamma
w}{w^{\alpha}\varGamma(1-\alpha)} \,
\mathrm{d} w
\nonumber
\\
& + 2 \eta\int_{0}^{\infty} (1 - \cos\gamma w)
\frac{e^{-\eta w}}{w^{\alpha}\varGamma(1-\alpha)} \mathrm{d} w
\Biggr\}
\nonumber
\\
={}& C_{\alpha, \eta} \hat{F} ( \gamma) 2 |\gamma| \int
_{0}^{\infty}
\frac{e^{-\eta w} \sin|\gamma| w}{w^{\alpha}\varGamma(1-\alpha)}
\, \mathrm{d} w
\nonumber
\\
={}&  C_{\alpha, \eta} \hat{F} ( \gamma) \frac{ 2| \gamma| }{ (
\eta^{2} + \gamma^{2})^{ \frac{1 - \alpha}{2}}} \sin\biggl((1-
\alpha) \arctan\frac{|\gamma|}{\eta} \biggr).
\nonumber
\end{align}
In the last step we used the following formula (\cite{gradtable}, p.~490, formula 5)
\begin{equation*}
\int_{0}^{\infty} x^{\mu- 1 } e^{ - \beta x}
\sin\delta x \, \mathrm{d} x = \frac{\varGamma(\mu)}{(\beta^{2} +
\delta^{2}) ^{ \frac{\mu}{2}}} \, \sin\biggl( \mu\arctan
\frac{\delta}{\mu} \biggr)
\end{equation*}
with $\mathrm{Re } \, \mu> -1$, $\mathrm{Re } \, \beta\geq\mathrm{Im
} \, \delta$.
%
\begin{remark}
For $ \eta\to0 $ we have that
\begin{align*}
\lim_{\eta\to0} \sin\biggl((1-\alpha) \arctan\frac{|\gamma|}{\eta}
\biggr) = \cos\biggl( \frac{\pi\alpha}{2} \biggr).
\end{align*}
Therefore
\begin{align*}
\lim_{\eta\to0} \int_{- \infty}^{+\infty}
e^{ i \gamma x} \, \frac{\partial^{\alpha, \eta} f}{\partial
|x|^{\alpha}} \, \mathrm{d} x= 2 C_{\alpha, 0} \, |
\gamma|^{\alpha}\cos\biggl( \frac{\pi\alpha}{2} \biggr)\, \hat
{F}(\gamma)
\end{align*}
and thus the normalizing constant must be $ C_{\alpha, 0} = - ( 2 \cos
\frac{\pi\alpha}{2} ) ^{-1} $.

This means that for $ \eta\to0 $ we obtain from \eqref
{eq:temp-riesz-fou} the Fourier transform of the Riesz fractional
derivative \eqref{eq:riesz-der}. This result shows that symmetric
stable processes are governed by equations
\begin{equation*}
\frac{\partial u}{\partial t } = \frac{\partial^{\alpha}u}{\partial
|x|^{\alpha}}
\end{equation*}
see, for example, \cite{toaldo14}, where the interplay between stable
laws, including subordinators and inverse
subordinators, and fractional equations is considered.
\end{remark}
%
%
\begin{remark}
For fractional equations of the form
%
\begin{equation}
\lleft\{
\begin{aligned}
&\frac{\partial u}{\partial t } = \frac{\partial^{\alpha, \eta}
u}{\partial|x|^{\alpha}},
& & t> 0, \;
x \in\mathbb{R},
\\
& u(x,0) = \delta(x), & & x \in\mathbb{R},
\end{aligned}
\rright.
\end{equation}
the Fourier transform of the solution reads
\begin{align*}
& \int_{-\infty}^{+\infty} e^{i \gamma x} u(x,t) \,
\mathrm{d} x
\\
&= \exp\biggl\{ t \, C_{\alpha, \eta} \frac{2 |\gamma| }{ ( \eta
^{2} + \gamma^{2})^{ \frac{1 - \alpha}{2} } } \sin\biggl( (1-\alpha
) \arctan\frac{|\gamma|}{\eta} \biggr) \biggr\}
\\
&= \exp\biggl\{ t \, C_{\alpha, \eta} \frac{2 |\gamma| }{ ( \eta
^{2} + \gamma^{2})^{1 - \frac{ \alpha}{2} } } \biggl[ |\gamma| \cos
\biggl( \alpha\arctan\frac{|\gamma|}{\eta} \biggr) - \eta\sin
\biggl( \alpha
\arctan\frac{|\gamma|}{\eta} \biggr) \biggr] \biggr\} .
\end{align*}
\end{remark}

\section{Fractional equations governing the drifted Brownian motion}
The law of the drifted Brownian motion started at $x$ satisfies the equations
\begin{equation*}
\frac{\partial u}{\partial t } = \frac{\partial^{2} u }{\partial
y^{2}} - \mu\frac{\partial u}{\partial y}, \quad t>0, y
\in\mathbb{R},
\end{equation*}
and
\begin{equation*}
\frac{\partial u}{\partial t } = \frac{\partial^{2} u }{\partial
x^{2}} + \mu\frac{\partial u}{\partial x}, \quad t>0,\; x
\in\mathbb{R}.
\end{equation*}
We show here that the drifted Brownian motion is related to time
fractional equations with tempered derivatives. Let us consider the process
%
\begin{equation}
B^{\mu}(t) = B(t) + \mu t + x, \quad\mu\in\mathbb{R},\; x \in
\mathbb{R}.
\end{equation}
The law $ u = u(x,y,t) $ of the process $ B^{\mu}$ is given by
%
\begin{equation}
\label{key} u ( x, y, t ) = \frac{
e^{ - \frac{(y-x-\mu t)^{2}}{4t} }
}{\sqrt{4 \pi t}} = \frac{ e^{ - \frac{ (y-x)^{2}}{ 4t} } }{\sqrt{4 \pi t}}
e^{ - \mu^{2} \frac{t}{4} + \frac{\mu}{2} (y-x)}, \quad t>0,\; x,y
\in\mathbb{R}.
\end{equation}
%
%
\begin{thm}
The law of $B^{\mu}$ solves the Cauchy problem
%
\begin{equation}
\label{eq:frac-eq-u-thm}
\lleft\{
\begin{aligned} &\mathscr{D} _{t}^{\frac{1}{2}, \eta} u +
\sqrt{\eta} \, u = a(x,y) \biggl( \frac{\partial u}{\partial x} + \sqrt
{\eta} \, u
\biggr), \quad t>0,\; x,y, \in\mathbb{R},
\\
& \textcolor{white} {\mathscr{D} _{t}^{\frac{1}{2}, \eta} u + \sqrt
{\eta}
\, u} = - a(x,y) \biggl( \frac{\partial u}{\partial y} - \sqrt{\eta}
\, u \biggr), \quad
t>0,\; x,y, \in\mathbb{R},
\\
&u(x,y,0) = \delta(x-y)
\end{aligned}
\rright.
\end{equation}
with
\begin{align*}
\eta= \frac{\mu^{2} }{4}.
\end{align*}
\end{thm}

\begin{proof}
We start by computing the Laplace--Fourier transform of the function
\begin{equation*}
g(x,y,t) = \frac{e^{ - \frac{ (y-x)^{2}}{4t}}}{\sqrt{4 \pi t }},
\end{equation*}
that is,
\begin{align*}
\hat{\tilde{ g}} (y,\xi, \lambda) &= \int_{0}^{\infty}
e^{ - \lambda t} \int_{ -\infty}^{+\infty} e^{i \xi x}
g(x,y,t) \, \mathrm{d} x \,\mathrm{d}t
\\
&= \int_{0}^{\infty}e^{- \lambda t} e^{i \xi y \,- \,\xi^{2} t}
\,\mathrm{d} t
\\
&= \frac{e^{i \xi y}}{\lambda+ \xi^{2}} .
\end{align*}

By using the fact that
%
\begin{align}
\tilde{g}(x,y,\lambda) = \frac{e^{-|y-x|\sqrt{\lambda}}}{2 \sqrt
{\lambda}} =
\lleft\lbrace
\begin{array}{@{}ll}
\displaystyle\frac{e^{-(y-x)\sqrt{\lambda}}}{2 \sqrt{\lambda}},
\quad y>x,\\[9pt]
\displaystyle\frac{e^{-(x-y)\sqrt{\lambda}}}{2 \sqrt{\lambda}},
\quad y \leq x,
\end{array}
\rright.
\label{lapBYUSING}
\end{align}
we now compute the double transform of $ a(x,y) \frac{ \partial g
}{\partial x} $.
\begin{align}
\label{eq:fou-lap-der} &\int_{-\infty}^{\infty} e^{ i \xi x}
\bigl[ \mathbh{1}_{(-\infty,y]}(x)- \mathbh{1} _{ (y, \infty)}(x)
\bigr]
\frac{\partial\tilde{g}}{\partial x} (x,y,\lambda) \, \mathrm{d} x
\\
&= \frac{1}{2} \Biggl( \int_{ -\infty}^{y}
e^{i \xi y} e^{-(y-x)\sqrt{\lambda}} \,\mathrm{d} x + \int_{y}^{\infty
}e^{i \xi y}
e^{-(x-y)\sqrt{\lambda}} \,\mathrm{d} x \Biggr)
\nonumber
\\
&= \frac{e^{i\xi y}}{2} \Biggl( \int_{0}^{\infty}e^{- i \xi x}
e^{-x \sqrt{\lambda}} \,\mathrm{d} x + \int_{0}^{\infty}e^{ i \xi x}
e^{-x \sqrt{\lambda}} \,\mathrm{d} x \Biggr)
\nonumber
\\
&= \frac{e^{i\xi y}}{2} \biggl( \frac{1}{i \xi+ \sqrt{\lambda}} +
\frac{1}{ - i \xi+ \sqrt{\lambda}} \biggr)
\nonumber
\\
&= \sqrt{\lambda}\frac{e^{i \xi y}}{\lambda+ \xi^{2}} = \sqrt
{\lambda}\hat{\tilde{g}} .
\nonumber
\end{align}
This implies, by inverting the Fourier transform, that
%
\begin{equation}
\label{eq:inv-fou-g} a(x,y) \frac{\partial\tilde{ g}}{\partial x}=
\sqrt{\lambda}\tilde{g}.
\end{equation}
We recall that
%
\begin{equation}
\int_{0}^{\infty} e^{-\lambda t} \mathscr{D}
^{\frac{1}{2}}_{t} g \, \mathrm{d} t = \sqrt{\lambda}\tilde{g},
\end{equation}
thus by inverting the Laplace transform in \eqref{eq:inv-fou-g} we obtain
%
\begin{equation}
\label{eq:g-eq} \mathscr{D} ^{\frac{1}{2}}_{t} g = a(x,y)
\frac{\partial g}{\partial x}
\end{equation}
and, by considering the same arguments (see \eqref{lapBYUSING}),
\begin{align*}
\mathscr{D} ^{\frac{1}{2}}_{t} g = -a(x,y) \frac{\partial g}{\partial y}.
\end{align*}
Returning to our initial problem, by using \eqref{eq:g-eq} and \eqref
{eq:temp-rl-def} we have that
\begin{align*}
\mathscr{D} _{t}^{\frac{1}{2}, \frac{\mu^{2} }{4}} u &= e^{- \frac{\mu
^{2} }{4} t } \mathscr{D}
^{\frac{1}{2}}_{t} \bigl( e^{ \frac{\mu^{2}t }{4}} u \bigr) -
\frac{\mu}{2} u
\\
& = e^{+\frac{\mu}{2} (y-x) \,- \,\frac{\mu^{2}}{4} t } \, \mathscr
{D} ^{\frac{1}{2}}_{t} g -
\frac{\mu}{2} u
\\
& = e^{- \frac{\mu^{2}}{4} t} \, e^{\frac{\mu}{2} (y-x) } \, \,
a(x,y) \frac{\partial g}{\partial x} -
\frac{\mu}{2} u
\\
&= a(x,y) \biggl( \frac{\partial u}{\partial x} + \frac{\mu}{2} u
\biggr) -
\frac{\mu}{2} u
\end{align*}
and
\begin{align*}
\mathscr{D} _{t}^{\frac{1}{2}, \frac{\mu^{2} }{4}} u &= e^{- \frac{\mu
^{2} }{4} t } \mathscr{D}
^{\frac{1}{2}}_{t} \bigl( e^{ \frac{\mu^{2}t }{4}} u \bigr) -
\frac{\mu}{2} u
\\
& = e^{+\frac{\mu}{2} (y-x) \,- \,\frac{\mu^{2}}{4} t } \, \mathscr
{D} ^{\frac{1}{2}}_{t} g -
\frac{\mu}{2} u
\\
& = -e^{- \frac{\mu^{2}}{4} t} \, e^{\frac{\mu}{2} (y-x) } \, \,
a(x,y) \frac{\partial g}{\partial y} -
\frac{\mu}{2} u
\\
&= a(x,y) \biggl( -\frac{\partial u}{\partial y} + \frac{\mu}{2} u
\biggr) -
\frac{\mu}{2} u.
\end{align*}
This completes the proof.
\end{proof}

The drifted Brownian motion has therefore a transition function satisfying
a time fractional equation where the fractional derivative is a
tempered Riemann--Liouville derivative with parameter $ \eta$ which is
related to the drift by the relationship $ \sqrt{\eta}= \frac{\mu}{2} $.

\section{Fractional equation governing the folded drifted Brownian motion}
We here consider the process
%
\begin{equation}
|B(t) + \mu t | + x = |B^{\mu}(t) | + x, \quad x>0.
\end{equation}
This process has distribution
\begin{align*}
P\bigl(| B(t) + \mu t| + x < y\bigr) ={}& P\bigl(x-y-\mu t < B(t) <
y-x-\mu t
\bigr)
\\
={}& \int_{x-y-\mu t}^{y-x-\mu t} \frac{e^{ - \frac{w^{2}}{4t }}}{\sqrt
{4 \pi t }} \,
\mathrm{d} w
\end{align*}
and therefore its transition function is
\begin{align}
\label{eq:v-tran-foo} P\bigl(|B ^{\mu}(t)| + x \in\mathrm{d} y\bigr) /
\mathrm{d} y &= \frac{
e^{ - \frac{(y-x-\mu t)^{2}}{4t} }
}{\sqrt{ 4 \pi t}} + \frac{
e^{ - \frac{(y-x+\mu t)^{2}}{4t} }
}{\sqrt{ 4 \pi t}}
\\
&= \frac{e^{ - \frac{(y-x)^{2}}{4t} }}{\sqrt{4 \pi t}} e^{ - \mu^{2}
\frac{t}{4} } \bigl[ e^{ - \frac{\mu}{2} (y-x) } +
e^{ \frac{\mu}{2} (y-x) } \bigr]
\nonumber
\\
&= v(x,y,t)
\nonumber
\end{align}
for $ y>x $ and $t>0$.
We now prove the following theorem.
%
\begin{thm}
The law $ v $ of $ | B^{\mu}(t) | +x $ satisfies the fractional equation
%
\begin{equation}
\label{eq:frac-eq-v-thm} \mathscr{D} _{t} ^{ \frac{1}{2} , \eta} v = -
\frac{\partial v }{\partial y} + v \, \sqrt{\eta}\tanh\bigl(\sqrt
{\eta}( y-x)\bigr) -
\frac{\mu}{2} v, \quad y > x > 0,
\end{equation}
with initial and boundary conditions
\begin{align*}
v(x,y,0) ={}& \delta( y - x),
\\
v(x,x,t) ={}& \frac{e^{-\eta t}}{\sqrt{\pi t}}, \quad t>0,
\end{align*}
and
\begin{align*}
\eta= \frac{\mu^{2}}{4}.
\end{align*}
\end{thm}
\begin{proof}
From \eqref{eq:v-tran-foo} and \eqref{eq:temp-rl-def} we have that
\begin{align*}
\mathscr{D} _{t}^{\frac{1}{2} } \bigl( e^{ \frac{\mu^{2} t}{4} } v
\bigr) &= 2
\cosh\biggl( \frac{\mu}{2} (y-x) \biggr) \,\, \mathscr{D}
_{t}^{ \frac{1}{2} } \biggl( \frac{e^{ - \frac{(y-x)^{2}}{4t} }}{\sqrt
{4 \pi t}} \biggr)
\end{align*}
Let $E_{\frac{1}{2}}$ be the Mittag-Leffler function of order $1/2$ and
$g$ be the function
\begin{align*}
g(x,y,t) = \frac{e^{ - \frac{(y-x)^{2}}{4t} }}{\sqrt{4 \pi t}}, \quad y>x>0.
\end{align*}
Since
\begin{align*}
\int_{0}^{\infty}e^{-\lambda t} \int
_{x}^{\infty}e^{-\xi y} g(x,y,t) \, \mathrm{d} y\,
\mathrm{d} t ={}& e^{-\xi x} \int_{0}^{\infty}e^{-\lambda t}
E_{\frac{1}{2}}\bigl(- \xi t^{\frac{1}{2}}\bigr) \, \mathrm{d} t
\\
={}& e^{-\xi x} \frac{\lambda^{\frac{1}{2}-1}}{\xi+ \lambda^{\frac{1}{2}}}
\end{align*}
we obtain that
\begin{align*}
\mathscr{D}^{\frac{1}{2}}_{t} g(x,y,t) = - \frac{\partial g}{\partial
y} \quad
\text{\textrm{with boundary condition }}  g(x,x, t) = \frac{1}{\sqrt
{4\pi t}}.
\end{align*}

Then, for $y>x$,
\begin{align*}
& \mathscr{D} _{t}^{\frac{1}{2} } \bigl( e^{ \frac{\mu^{2} t}{4} } v
\bigr)
\\
& = 2 \cosh\biggl( \frac{\mu}{2} (y-x) \biggr) \,\, \biggl( -
\frac{\partial}{\partial y} \biggl( \frac{e^{ - \frac{(y-x)^{2}}{4t}
}}{\sqrt{4 \pi t}} \biggr) \biggr)
\\
&= -\frac{\partial}{\partial y} \biggl( 2 \cosh\biggl( \frac{\mu}{2}
(y-x) \biggr)
\frac{e^{ - \frac{(y-x)^{2}}{4t} }}{\sqrt{4 \pi t}} \biggr) + \frac{\mu
}{2} \cdot2 \sinh\biggl(
\frac{\mu}{2} (y-x) \biggr) \frac{e^{ - \frac{(y-x)^{2}}{4t} }}{\sqrt
{4 \pi t}}
\\
&= -\frac{\partial}{\partial y} \bigl( e^{\frac{\mu^{2}}{4}t }
v(x,y,t) \bigr) +\mu\, \sinh
\biggl( \frac{\mu}{2} (y-x) \biggr) \frac{e^{ - \frac{(y-x)^{2}}{4t}
}}{\sqrt{4 \pi t}}
\end{align*}
with boundary condition
\begin{align*}
e^{\frac{\mu^{2}}{4}t} v(x,x,t)= \frac{1}{\sqrt{4\pi t}} + \frac
{1}{\sqrt{4\pi t}} .
\end{align*}
In view of \eqref{eq:temp-rl-def} we obtain that
\begin{align*}
\mathscr{D} ^{\frac{1}{2}, \eta}_{t} v + \frac{\mu}{2} v & =
e^{ - \frac{ \mu^{2}}{4}t } \mathscr{D} _{t}^{\frac{1}{2}} \bigl(
e^{ \frac{\mu^{2}}{4}t } v \bigr)
\\
&= -\frac{\partial v}{\partial y} - \mu e^{ - \frac{\mu^{2}}{4}t }
\sinh\biggl( \frac{\mu}{2}
(y-x) \biggr) \frac{e^{ - \frac{(y-x)^{2}}{4t} }}{\sqrt{4 \pi t}}
\\
&= -\frac{\partial v}{\partial y} - \frac{\mu}{2} \, \tanh\biggl(
\frac{\mu}{2}
(y-x) \biggr) v.
\qedhere
\end{align*}
\end{proof}

This result shows that the structure of the governing equation of the
process $ |B(t) + \mu t | + x $ is substantially different from that of
$ B(t) + \mu t + x $.
The difference between
\eqref{eq:frac-eq-u-thm} and \eqref{eq:frac-eq-v-thm}
consists in the non-constant coefficient $ \tanh\frac{\mu}{2} (y-x) $
which converges to one as the difference $ |y-x| $ tends to infinite.
Thus the two equations emerging in this analysis coincide for
$ |x-y| \to\infty$.




\end{document}